\documentclass[11pt,reqno,a4paper]{article}
\usepackage{amsmath,amssymb,amsthm,amsfonts,epsfig,enumerate}
\usepackage{mathtools}
\usepackage{multicol}
\usepackage{mathrsfs,amssymb,soul}
\usepackage[utf8]{inputenc}
\usepackage{amsmath,amssymb,amsthm,amsfonts,epsfig,enumerate}
\usepackage{pst-plot,pst-node}

\allowdisplaybreaks[4] 

\usepackage{graphicx}
\usepackage{xcolor}
\usepackage{hyperref}
\hypersetup{
    colorlinks=true,
    linkcolor=myblue,
    filecolor=magenta,      
    urlcolor=cyan,
    citecolor=mygreen,
}
\definecolor{mygreen}{rgb}{0.0, 0.6, 0.1}
\definecolor{myblue}{rgb}{0.01, 0.18, 1.0}
\newtheorem{definition}{Definition}[section]
\newtheorem{theorem}[definition]{Theorem}
\newtheorem{corollary}[definition]{Corollary}
\newtheorem{proposition}[definition]{Proposition}
\newtheorem{Lemma}[definition]{Lemma}

\theoremstyle{definition}
\newtheorem{remark}[definition]{Remark}

\numberwithin{equation}{section}
\DeclarePairedDelimiter\abs{\lvert}{\rvert}%
\DeclarePairedDelimiter\norm{\lVert}{\rVert}%
\makeatletter
\let\oldnorm\norm
\def\norm{\@ifstar{\oldnorm}{\oldnorm*}}

\newcommand{\pa} {\partial}

\newcommand{\De} {\Delta}

\newcommand{\ga} {\gamma}
\newcommand{\Ga} {\Gamma}
\newcommand{\om} {\omega}
\newcommand{\Om} {\Omega}

\newcommand{\la} {\lambda}

\newcommand{\Gr} {\nabla}
\newcommand{\no} {\nonumber}
\newcommand{\noi} {\noindent}
\newcommand{\vph} {\varphi}

\newcommand{\ep} {\epsilon}
\newcommand{\ra} {\rightarrow}

\newcommand\restr[2]{{
  \left.\kern-\nulldelimiterspace 
  #1 
  \right|_{#2} 
  }}

\def\inpr#1{\left\langle #1\right\rangle}
\def\wp{{W^{1,p}(\Om)}}
\def\wRN{{W^{1,p}(\RN)}}
\def\wpb{{W^{1,p}(B_1^c)}}
\def\wNb{{W^{1,N}(B_1^c)}}

\def\A{{\mathcal A}}
\def\C{{\mathcal C}}
\def\D{{\mathcal D}}

\def\P{{\Phi}}
\def\M{{\mathcal M}}

\def\R{{\mathbb R}}
\def\RN{{\mathbb R}^N}
\def\N{{\mathbb N}}
\def\F{{\mathcal F}}

\def\({{\Big(}}
\def\){{\Big)}}

\def\ws2{{\F_{\frac{N}{2}}}}
\def\l2{{ L^{1,\\infty}(\log L)^2}}
\def\M2{\mathcal M_{\log L}}
\def\cc{{\C_c^\infty}}
\def\dis{{\displaystyle \int_{\Omega}}}
\def\disb{{\displaystyle \int_{B^c_1}}}

\def\resb{{C_c^1(\overline{B_1^c})}}
\def\p{{p^{\prime}}}

\def\Np{{N^{\prime}}}
\def\g{{\tilde{g}}}
\def\dr{{\rm d}r}
\def\dS{{\rm d}S}
\def\ds{{\rm d}s}
\def\dt{{\rm d}t}
\def\dx{{\rm d}x}
\def\dJ{{\rm d}J}

\usepackage[hyperpageref]{backref}

\title{Neumann eigenvalue problems  on the exterior domains}
\author{T. V. Anoop\footnote{corresponding author and also supported by the INSPIRE Research Grant DST/INSPIRE/04/2014/001865.}\;,\;
Nirjan Biswas}
\date{}
\begin{document}
\maketitle
\begin{abstract}
For $ p\in (1, \infty)$, we consider the following weighted Neumann eigenvalue problem on $B_1^c$, the exterior of the closed unit ball in $\R^N$:
\begin{equation}\label{Neumann eqn}
\begin{aligned}
-\Delta_p \phi & = \la g |\phi|^{p-2} \phi \ \text{in}\ B^c_1,\\
 \displaystyle\frac{\pa \phi}{\pa \nu} &= 0  \ \text{on} \ \pa B_1, 
\end{aligned}
\end{equation}
where $\De_p$ is the $p$-Laplace operator  and  $g \in L^1_{loc}(B^c_1)$ is an indefinite weight function. Depending on the values of $p$ and the dimension $N$,  we take $g$ in certain Lorentz spaces or weighted Lebesgue spaces and show that \eqref{Neumann eqn}  admits an unbounded sequence of positive eigenvalues that includes a unique principal eigenvalue.  For this purpose,  we establish the compact embeddings of $ \wpb$ into $L^p(B^c_1, |g|)$ for $g$ in certain weighted Lebesgue spaces.  For $N>p$, we also provide an alternate  proof for the embedding of $\wpb$ into $L^{p^*,p}(B^c_1)$. Further, we  show that the set of all eigenvalues of \eqref{Neumann eqn} is closed. 
\end{abstract}

\section{Introduction}
Let $\Om$ be a smooth domain in $\R^N$ and $ g \in L_{loc}^1(\Om) $. For $p \in (1, \infty),$ we consider the following nonlinear weighted eigenvalue problem:
\begin{equation}\label{Eqn:Neumann}
\begin{aligned}
-\Delta_p \phi & = \lambda g |\phi|^{p-2} \phi \ \mbox{in}\ \Om,\\
\frac{\pa \phi}{\pa \nu} &= 0  \ \mbox{on} \ \pa \Om, 
\end{aligned}
\end{equation}
where $\Delta_p$ is the $p$-Laplace operator defined as $\Delta_p(\phi) =  {\rm div} (|\Gr \phi|^{p-2} \Gr \phi)$. We say a real number $\la$ is an eigenvalue of \eqref{Eqn:Neumann}, if there exits $ \phi \in \wp \setminus \{0\} $ satisfying the following: 
\begin{align}\label{weak}
   \dis  |\Gr \phi|^{p-2} \Gr \phi \cdot \Gr \upsilon = \la \dis g |\phi|^{p-2} \phi \upsilon, \quad
   \forall \upsilon \in \wp.
\end{align}
In this case, we also say $\phi$ is  an eigenfunction of \eqref{Eqn:Neumann} corresponding to $\la$. An eigenvalue $\la$ is called a {\it principal eigenvalue}, if there exits an eigenfunction corresponding to $\la$ that does not change sign on $\Om$. 

  If $\Om$ is bounded, then zero is always a principal eigenvalue of \eqref{Eqn:Neumann} (nonzero constants as corresponding eigenfunctions). If  $\int_{\Om} g(x) \dx \geq 0$, then zero is the only nonnegative principal eigenvalue. Thus when $\Om$ is bounded, for the existence of a positive principal eigenvalue of \eqref{Eqn:Neumann}, $ \int_{\Om} g(x) \dx < 0 $ is necessary. This condition alone does not ensure the existence of  a positive principal eigenvalue for  \eqref{Eqn:Neumann}. Under the additional assumptions such as $g \in L^{\infty}(\Om)$ (\cite{BrownLin}), $g \in C(\bar{\Om})$ (\cite{Huang}), or $g \in L^{d}(\Om)$ with $ d > \frac{N}{p}$ (\cite{HabibTsouli}), the eigenvalue problem \eqref{Eqn:Neumann} does admit a principal eigenvalue and it is unique. 

 If $\Om=\RN,$ then \eqref{weak} corresponds to the weak formulation of the Dirichlet eigenvalue problem. In this context, for $N>p$, \eqref{Eqn:Neumann} admits a positive principal eigenvalue even for certain $g$ with $\int_{\Om} g(x) \dx \ge 0 $.  For example, smooth $g$ with $g^-$ is bounded away from zero at infinity (\cite{Brown-Cosner-Fleckinger, Huang1995}), $g$ with  $g^+ \in L^{\frac{N}{2}}(\RN)$ (\cite{Allegretto,Allegretto-Huang}).  Further, if the eigenfunctions  are allowed to be in Beppo-Levi space $\D^{1,p}_0(\R^N):=$ completion of $\cc(\R^N)$ with respect to the $\norm{\Gr\, \cdot\,}_p,$ then \eqref{Eqn:Neumann} admits a positive principal eigenvalue for weights in bigger classes of function spaces, see \cite{Anoop, AMM} and the references therein.  For $N=p$, if $\int_{\R^N} g(x) \dx > 0,$  the non-existence of positive principal eigenvalue for \eqref{Eqn:Neumann} is proved  in \cite{Brown-Cosner-Fleckinger} for $p=2$ and in \cite{Huang1995} for general $p$.

 In this article, we study the existence of a positive principal eigenvalue of \eqref{Eqn:Neumann} on  $\Om=B_1^c$. The Dirichlet eigenvalue problem for $p$-Laplacian on the exterior domain is considered in \cite{ADS}.   We enlarge the class of weight functions that admits a positive principal eigenvalue by providing two distinct categories of function spaces. The first category contains certain closed subspace of the  Lorentz space $ L^{\frac{N}{p}, \infty}(B^c_1)$  for  $N>p$ and the second one  contains certain weighted Lebesgue spaces for all choices of  $p$ and $N$. 
 
We consider the following closed subspace (introduced in \cite{AMM}) of  the  Lorentz space $ L^{\frac{N}{p}, \infty}(B^c_1):$
$$ \F_{\frac{N}{p}} := \mbox{closure of} \ C_c^{\infty}(B^c_1) \ \mbox{in} \ L^{\frac{N}{p},\infty}(B^c_1).$$ For details of the space $\F_{\frac{N}{p}}$, we refer to \cite{AMM}. 
\begin{theorem}\label{Lorentz space}
Let $p\in (1,\infty)$ and $N > p$. If $g \in \F_{\frac{N}{p}}$ and $g^+ \not \equiv 0$, then $$ \la_1 = \inf \left\{ \disb |\Gr \phi|^p : \phi \in \wpb, \disb g |\phi|^p = 1 \right\} $$ is the unique positive principal eigenvalue of \eqref{Eqn:Neumann}. Furthermore, $\la_1$ is simple and isolated.
\end{theorem}

Our proof for the above theorem  uses the continuous embedding of the Sobolev space $\wpb$ into the Lorentz space $L^{p^*,p}(B^c_1)$. This embedding can be obtained from  the embedding of  $W^{1,p}(\R^N)$ into $L^{p^*,p}(\R^N)$ due to Tartar \cite{Tartar}. However, we give a simple proof for this embedding using the P\'{o}lya-Szeg\"{o} and the Hardy-Littlewood inequalities for the  Schwarz symmetrization and the Muckenhoupt condition (Theorem 2 of \cite{Muckenhoupt}) for the one-dimensional weighted Hardy inequalities.

 To state our next result, we associate a radial function with $g$ as below. For $ r \in [1, \infty),$
$$ \g(r) = \mbox{ess}\sup \{ |g(r\om)|: \om \in S^{N-1} \},$$ where the essential supremum is taken with respect to the $(N-1)$-dimensional surface measure. Since $g \in L^1_{loc}(B^c_1)$, we get $\g(r)$ is finite a.e. in $(1, \infty)$ (Theorem 2.49 of \cite{Folland}). Now we consider the following weighted Lebesgue spaces:
$${X}= \left\{\begin{array}{ll}
                    L^1((1, \infty); r^{p-1}), \, & N \neq p,\\
                     L^1((1, \infty); {(r(1 + \log r))}^{N-1}), \, &N=p.
                   \end{array}\right.$$
\begin{theorem}\label{Exterior Ball}
Let $p \in (1, \infty)$ and let $g \in L^1_{loc}(B^c_1)$ with $g^+ \not \equiv 0$.  If  $\g \in X $, then $$ \la_1 = \inf \left\{ \disb |\Gr \phi|^p : \phi \in \wpb, \disb g |\phi|^p = 1 \right\} $$
 is the unique positive principal eigenvalue of \eqref{Eqn:Neumann}. Furthermore, $\la_1$ is simple and isolated. 
\end{theorem}
The Dirichlet eigenvalue problem for $g$ for which $\g$ lies in an analogous weighted Lebesgue space has been considered in \cite{ADS}.  For $\g \in X$, we show that $\wpb$ is continuously  and compactly embedded into the weighted Lebesgue space $L^p(B_1^c, |g|).$ A similar embedding for $\D^{1,p}_0(B^c_1)$ is obtained in \cite{ADS}.

We also study the existence of infinitely many positive eigenvalues of \eqref{Eqn:Neumann}. A complete characterization of the set of all eigenvalues of \eqref{Eqn:Neumann} with $p\ne 2$  is a challenging open problem. However, there are many ways to produce infinite set of eigenvalues of \eqref{Eqn:Neumann}, for example, see \cite{HabibTsouli, AnLe}. In \cite{AnLe},  An L\^{e}  proved that, for $\Om$ bounded and $g \equiv 1$, the set of all eigenvalues of \eqref{Eqn:Neumann} is closed. We extend these results as below:
\begin{theorem}\label{infinite eigenvalues}
Let $p \in (1, \infty)$ and $g$ be in Theorem \ref{Lorentz space} or  Theorem \ref{Exterior Ball}.
Then 
\begin{enumerate}[(i)]
 \item there exists a sequence of positive eigenvalues of \eqref{Eqn:Neumann} tending to infinity,
 \item the set of all eigenvalues of \eqref{Eqn:Neumann} is closed.
\end{enumerate}
 \end{theorem}

The rest of the paper is organized as follows. In Section 2, we briefly define  symmetrization, Lorentz spaces and state the Muckenhoupt conditions for the weighted Hardy inequalities. In Section 3, we prove the required continuous embeddings and its compactness. Section 4 contains the functional settings. In the last section, we give the proofs of the above theorems.
\section{Preliminaries}\label{preliminary}
We define the one-dimensional rearrangement and then define the Lorentz spaces. Further, we state some important results such as  Muckenhoupt condition, maximum principle for $p$-Laplacian that will be used subsequently. 
\subsection{Symmetrization}
Let $ \Om \subset \R^N $ be a Lebesgue measurable set and let  $\mathcal{M}(\Om)$ be the set of all extended real valued Lebesgue measurable functions that are finite a.e. in $\Om$. Given a function $f \in \mathcal{M}(\Om)$ and for $s > 0, $ we define $E_{f}(s) = \{ x \in \Om : |f(x)| > s \}.$  The \textit{distribution function} $ \alpha_{f} $ of $f$ is defined as
$
\alpha_{f}(s) = |E_{f}(s)| \; \text{for} \; s > 0,
$
where $ |\cdot| $ denotes the Lebesgue measure. We define the \textit{one dimensional decreasing rearrangement} $ f^{*} $ of $f$ as
\begin{equation*}
		f^*(t) = \inf \{ s > 0 : \alpha_{f}(s) < t \}, \; \mbox{ for } t > 0. 
\end{equation*}
The map $f \mapsto f^*$ is not sub-additive. However, we obtain a sub-additive function from $f^*,$ namely the maximal function $f^{**}$ of $f^*$, defined by 
 \begin{equation*}
 f^{**}(t)=\frac{1}{t}\int_0^tf^*(\tau)\, {\rm d}\tau, \quad t>0.
 \end{equation*}
The { \it Schwarz symmetrization } of $f$ is defined by 
 \begin{equation*}
 f^\star(x)=f^*(\omega_N|x|^N)  \label{relation}, \quad \forall\, x\in \Omega^\star,
 \end{equation*}
 where $\omega_N$ is the measure of the unit ball in $\R^N$ and $\Omega^\star$ is the open ball centered at the origin with same measure as $\Omega.$

Next we state two important inequalities concerning the  symmetrization. For more details we refer to the books \cite{HL, PS}.
\begin{proposition}\label{HL and PS}
 Let $\Omega \subset \R^N$ with $N\ge 2$.
 \begin{enumerate}[(a)]
 \item Hardy-Littlewood inequality: Let $f$ and $g$ be nonnegative measurable functions. Then
 $$\int_{\Omega} f(x)g(x) \dx \leq \int_{\Omega^\star}f^\star(x)g^\star(x) \dx = \int_0^{|\Omega|}f^*(t) g^*(t)\dt.$$
 \item P\'{o}lya-Szeg\"{o} inequality: 
Let $\phi \in \wRN$. Then  
$$  \int_{\Omega^\star}|\Gr \phi^\star(x)|^p \dx = \displaystyle N^p \om_N^{\frac{p}{N}} \int_{0}^{\infty} s^{(p - \frac{p}{N})} | {\phi^*}^{\prime}(s)|^p ds \leq \int_{\RN} | \Gr \phi(x)|^p \dx. $$
 \end{enumerate}	
\end{proposition}

\subsection{Lorentz Space}
The Lorentz spaces are introduced by Lorentz in \cite{Lorentz} and these are refinements of the classical Lebesgue spaces. For more details on Lorentz spaces, we refer to the books \cite{Adams,EdEv}. 

Let $\Om \subset \R^N$ be an open set. Let $f \in \mathcal{M}(\Omega)$ and $(p,q) \in [1,\infty)\times[1,\infty]$. Consider the following quantity:
 \begin{align*} 
 |f|_{(p,q)} := \norm{t^{\frac{1}{p}-\frac{1}{q}} f^{*} (t)}_{{L^q((0,\infty))}}
 =\left\{\begin{array}{ll}
 \left(\displaystyle\int_0^\infty \left[t^{\frac{1}{p}-\frac{1}{q}} {f^{*}(t)}\right]^q \dt \right)^{\frac{1}{q}}, \; 1\leq q < \infty; \vspace{4mm}\\ 
 \displaystyle\sup_{t>0}t^{\frac{1}{p}}f^{*}(t), \; q=\infty.
 \end{array} 
 \right.
 \end{align*}
The Lorentz space $L^{p,q}(\Om)$ is defined as
 \[ L^{p,q}(\Om) := \left \{ f\in \mathcal{M}(\Om): \,   |f|_{(p,q)}<\infty \right \},\]
 where $ |f|_{(p,q)}$ is  a complete quasi norm on $L^{p,q}(\Om).$ For  $p\in (1,\infty)$, $L^{p,p}(\Omega)=L^p(\Omega)$ and  $ L^{p,\infty}(\Om)$ coincides with the weak-$L^p$ space 
 $:=\left \{ f \in \mathcal{M}(\Om):\sup_{s>0}s(\alpha_f(s))^{\frac{1}{p}}<\infty \right \}.$ 
 Indeed, one can define a norm on $L^{p,q}(\Om)$ for certain values of $p$ and $q$ as in the following proposition (Lemma 3.4.6 of \cite{EdEv}).
\begin{proposition}\label{equivallence}
 For $(p,q) \in (1,\infty)\times[1,\infty]$, let  
 $$ \norm{f}_{(p,q)}:= \norm{t^{\frac{1}{p}-\frac{1}{q}} f^{**} (t)}_{{L^q((0,\infty))}}.$$  Then $\norm{f}_{(p,q)}$ is a norm in $L^{p,q}(\Om)$ and it is equivalent to the quasi-norm $|f|_{(p,q)}$.
\end{proposition} 
\subsection{Some important results}
The following result is a sufficient condition for the one-dimensional weighted Hardy inequalities (4.17 of \cite{KMP}).
\begin{proposition}[Muckenhoupt condition]\label{Mucken}
Let $u,v$ be nonnegative measurable functions such that $v > 0$. Let $p\in(1,\infty)$ and let $\p$ be the H\"older conjugate of $p.$ If
$$ \displaystyle A = \sup_{t > 0} \left( \int_{0}^t u(s) \ds \right)^{\frac{1}{p}} \left( \int_{t}^{\infty} v(s)^{1 - \p} \ds \right)^{\frac{1}{\p}} < \infty, $$
then 
\begin{align}\label{Mucken1}
\displaystyle \left( \int_{0}^{\infty} \left | \int_{s}^\infty f(t) dt \right|^p u(s) \ds \right)^{\frac{1}{p}} \leq p^{\frac{1}{p}} (\p)^{\frac{1}{\p}}A \left( \int_{0}^\infty | f(s) |^p v(s) \ds \right)^{\frac{1}{p}}
\end{align}
holds for any measurable function $f$ on $(0,\infty).$ 
\end{proposition} 
In this article we use the following version of strong maximal principle due to Kawohl, Lucia and Prashanth (Proposition 3.2 of \cite{Kawohl-Lucia-Prashanth}). 
\begin{proposition}[Strong Maximum Principle for $p$-Laplacian]\label{STM}
Let $\phi$ be a non negative function in $\wpb$ and $V \in L^1_{loc}(B^c_1)$ with $V \geq 0$ a.e. in $B^c_1$. Assume that $V\phi^{p-1} \in L^1_{loc}(B^c_1)$. Consider the inequality
$$  \disb  |\Gr \phi|^{p-2} \Gr \phi \cdot \Gr \upsilon + \disb V |\phi|^{p-2} \phi \upsilon \geq 0, \quad \forall v \in C_c^{\infty}(B^c_1), v \geq 0. $$
Then either $\phi \equiv 0$ or $\phi > 0$ a.e. in $B^c_1$.
\end{proposition}

\section{The embeddings of $\wpb$ }
For $N>p$, we prove the  continuous embeddings of $\wpb$ into $L^{p^*, p}(B^c_1)$, where $p^* = \frac{Np}{N-p}.$ For $g$ as in Theorem \ref{Exterior Ball}, we prove $\wpb$ is continuously and compactly embedded in $L^p(B^c_1,|g|)$.
\subsection{The embeddings into Lorentz spaces}
First we prove a lemma using the Muckenhoupt condition.

\begin{Lemma}\label{For embedding}
Let $N > p$. If $g \in L^{\frac{N}{p}, \infty}(\RN)$, then 
\begin{align}\label{for embedding}
\displaystyle \int_0^{\infty} g^*(s) {\phi^*(s)}^p \ds  \leq  p(p^*)^{p-1} \| g \|_{(\frac{N}{p}, \infty)}  \int_{0}^{\infty} s^{(p - \frac{p}{N})} | {\phi^*}^{\prime}(s) |^p \ds,\quad \forall \phi \in C_c^1(\RN).
\end{align}
\end{Lemma}
\begin{proof}
In Proposition \ref{Mucken},  set $f = {\phi^*}^{\prime}, u = g^*$ and $v(s)= s^{p - \frac{p}{N}}.$ Then $$\int_{0}^t u(s) \ds=\int_0^t g^*(s) \ds = tg^{**}(t)$$ and
\begin{align*}
\displaystyle \int_t^{\infty} {v(s)}^{1 - \p} \ds = \int_t^{\infty} s^{\frac{-\p}{\Np}}\ds = \frac{\Np}{\p - \Np} t^{\frac{\Np - \p}{\Np}} =  \frac{N(p-1)}{N-p} t^{\frac{p-N}{N(p-1)}}.
\end{align*}
Now,
\begin{align*}
 A &= \sup_{t>0} \left( \int_{0}^t u(s) \ds \right)^{\frac{1}{p}} \left( \int_{t}^{\infty} v(s)^{1 - \p} \ds \right)^{\frac{1}{\p}}\\ 
& = \left( \frac{N(p-1)}{N-p} \right)^{\frac{1}{\p}} \sup_{t>0} \{ t g^{**}(t)\}^{\frac{1}{p}} t^{\frac{p-N}{Np}}
= \left( \frac{N(p-1)}{N-p} \right)^{\frac{1}{\p}} \|g \|_{(\frac{N}{p}, \infty)}^{\frac{1}{p}}. 
\end{align*}
Therefore, by the Muckenhoupt condition  we have for all $\phi \in C_c^1(\RN),$
$$ \displaystyle \left(  \int_0^{\infty} g^*(s) {\phi^*(s)}^p \ds \right)^{\frac{1}{p}}  \leq p^{\frac{1}{p}}  \left( \frac{N(p-1)\p}{N-p} \right)^{\frac{1}{\p}} \| g \|^{\frac{1}{p}}_{(\frac{N}{p}, \infty)}  \left( \int_{0}^{\infty} s^{(p - \frac{p}{N})} | {\phi^*}^{\prime}(s)|^p \ds\right)^{\frac{1}{p}}.$$
Now \eqref{for embedding} follows by noting that  $p(p^*)^{p-1}$ is precisely the $p^ {\rm th}$ power of the constant in the right hand side of the above inequality.
\end{proof}
\begin{theorem}\label{embedding}
For $N > p,$ there exists $C>0$ such that \begin{equation}\label{Lorentz embedding for RN}
\| \phi \|_{(p^*,p)}^p \leq C \|\phi \|^p_{\wRN},\quad \forall\phi\in \wRN.
\end{equation}
\end{theorem}
\begin{proof}
Let $g\in L^{\frac{N}{p}, \infty}(\RN).$ Then by P\'{o}lya-Szeg\"{o} inequality (part (b) of Proposition \ref{HL and PS}) and by the above lemma we have
\begin{align}\label{lorentz inq1}
\displaystyle \int_0^{\infty} g^*(s) {\phi^*(s)}^p \ds & \le \tilde{C} \, \| g \|_{(\frac{N}{p}, \infty)} \int_{\RN} |\Gr \phi(x)|^p \dx, \quad \forall \phi \in C_c^1(\RN),
\end{align}
with $\tilde{C} = p(p^*)^{p-1}N^{-p} \om_N^{-\frac{p}{N}}$. As $g(x) = \frac{1}{|x|^p}$ is in $L^{\frac{N}{p}, \infty}(\RN)$ with $g^*(s) = \left( \frac{\om_N}{s} \right)^{\frac{p}{N}}$ and 
$\| g \|_{(\frac{N}{p}, \infty)} = \frac{N \om_N^{\frac{p}{N}}}{N - p},$ from \eqref{lorentz inq1} we have
\begin{align*}
 \int_0^{\infty} s^{- \frac{p}{N}} {\phi^*(s)}^p \ds & \leq  \tilde{C} \frac{N}{N-p} \; \int_{\RN} |\Gr \phi(x)|^p \dx \leq \tilde{C} \frac{N}{N-p} \| \phi \|_{\wRN}^p, \quad \forall\phi \in C_c^1(\RN).
\end{align*}
The integral in the left hand side of the inequality is equivalent to  $\| \phi \|_{(p^*,p)}$ (Proposition \ref{equivallence}) and hence by the density of  $C_c^1(\RN)$  in $\wRN$ we obtain  \eqref{Lorentz embedding for RN}  with $C= D \left( \frac{p^*}{N \om_N^{\frac{1}{N}}} \right)^{p},$ where $D$ is the equivalence constant. 
\end{proof}
\begin{remark}
For $g \in L^{\frac{N}{p}, \infty}(\RN)$, from \eqref{lorentz inq1} and Hardy-Littlewood inequality, we have the following generalized Hardy-Sobolev inequality: 
$$ \int_{\RN} g(x) |\phi(x)|^p \dx \leq \frac{p(p^*)^{p-1}}{N^p \om_{N}^{\frac{p}{N}}} \| g \|_{(\frac{N}{p}, \infty)} \int_{\RN} | \Gr \phi(x) |^p \dx, \quad \forall \phi \in \wRN.$$ In particular, by taking $g(x) = \frac{1}{|x|^p}$, we get the classical Hardy-Sobolev inequality $$ \int_{\RN} \frac{|\phi(x)|^p}{|x|^p} \dx \leq \left( \frac{p}{N-p} \right)^p \int_{\RN} | \Gr \phi(x) |^p \dx, \quad \forall \phi \in \wRN. $$
\end{remark}
\begin{corollary}\label{Lorentz embedding}
Let $N > p$. Then $ \Vert \phi \Vert_{(p^*,p)} \leq C \| \phi \|_{\wpb}. $
\end{corollary}
\begin{proof}
Since the boundary of $B^c_1$ is smooth, it has the extension property (Theorem 9.7 of \cite{Brezis}, page 272), i.e., there exists a positive constant $C$ such that $$ \|\phi \|_{\wRN} \leq C \| \phi \|_{\wpb}. $$ Now by Theorem \ref{embedding}, we get the required embedding.  
\end{proof}
\subsection{The embeddings into weighted Lebesgue spaces}\label{continuous embedding}
\begin{theorem}\label{weighted Hardy subcritical}
Let $N > p$. If $\g \in L^1((1, \infty);r^{p-1}),$ then there exits $C = C(N,p)>0$ such that  
\begin{align}\label{Hardy subcritical}
\disb \g(|x|)|\phi(x)|^p  \dx \leq C \Vert \g \Vert_{L^1((1, \infty) ; r^{p-1})} \| \phi \|_{\wpb}^p, \quad \forall \phi \in \wpb. 
\end{align}
\end{theorem}
\begin{proof}
Let $\psi \in \resb$. For $\om \in S^{N-1}$, set $\vph(t) = \psi(t \om)$ where $t \geq 1$. Using the fundamental theorem of calculus we have
\begin{align*}
\vph(r) = - \displaystyle \int_{r}^{\infty} \vph'(t) \dt = - \displaystyle \int_{r}^{\infty} \vph'(t) t^{\frac{1-N}{p}} t ^{\frac{N-1}{p}} \dt.
\end{align*}
By H\"{o}lder inequality,
\begin{align*}
| \vph(r)|^p \leq \left( \frac{p-1}{N-p} \right)^{p-1} r^{p-N} \left( \int_r^{\infty} t^{N-1} | \vph'(t) |^p \dt \right).
\end{align*}
 As $\vph'(t)= \Gr \psi(t \om) \cdot \om$, for each $\om \in S^{N-1}$ we have  $ |\vph'(t)| = |\Gr \psi(t \om)| $. Hence
\begin{align*}
| \psi(r\om)|^p \leq \left( \frac{p-1}{N-p} \right)^{p-1} r^{p-N} \left( \int_r^{\infty} t^{N-1} | \Gr \psi(t \om)|^p \dt \right).
\end{align*}
Set $C = \left( \frac{p-1}{N-p} \right)^{p-1}$. We multiply both sides by $r^{N-1} \g(r)$ and integrate over $S^{N-1} \times (1, \infty)$ to get
\begin{align*}
  \displaystyle  \int_{1}^{\infty} \int_{S^{N-1}} &|\psi(r\om)|^p r^{N-1} \g(r) \dS\, \dr 
 \leq C \int_{1}^{\infty} r^{p-1} \g(r) \dr \left(  \int_1^{\infty}  \int_{S^{N-1}}  r^{N-1} | \Gr \psi(r \om)|^p  \dS \; \dr \right).
\end{align*}
Thus we obtain
\begin{align*}
\disb \g(|x|) |\psi(x)|^p \dx  \leq C \left( \int_1^{\infty} r^{p-1} \g(r) \dr \right) \| \psi \|^p_{\wpb}, \quad \forall \psi \in \resb.
\end{align*}
Now \eqref{Hardy subcritical} follows by the density of $\resb$ in $\wpb$.
\end{proof}

\begin{theorem}\label{weighted Hardy critical}
Let $N = p$. If $\g \in L^1((1, \infty); {(r(1 + \log r))}^{N-1})$, then there exits $C = C(N) >0$ such that 
\begin{align*}
\no \disb  \g(|x|)\,|\phi(x)|^N  \dx \leq C & \| \g \|_{ L^1((1, \infty); {(r(1 + \log r))}^{N-1})} 
\| \phi \|_{\wNb}^N,\quad \forall\, \phi \in \wNb.
\end{align*}
\end{theorem}
\begin{proof} Let $\psi \in \resb$. As before, set $\vph(t) = \psi(t \om)$ where $t \geq 1$. Then
\begin{align*}
\vph(r) - \vph(1) = \int^r_1 \vph'(t) \dt  = \int^r_1 t^{\frac{1-N}{N}} t ^{\frac{N-1}{N}} \vph'(t) \dt.
\end{align*}
By  H\"{o}lder inequality,
\begin{align*}
\abs{\vph(r) - \vph(1)} \leq \left( \int_1^r \frac{1}{t} \dt \right)^{\frac{1}{\Np}} \left( \int_1^r t^{N-1} |\vph'(t)|^N \dt \right)^{\frac{1}{N}},
\end{align*}
where $\Np$ is the H\"older conjugate of $N.$ Thus
\begin{align*}
|\vph(r)|^N \leq 2^{N-1} \left\{ |\vph(1)|^N + (\log r)^{N-1} \int_1^r t^{N-1} |\vph'(t)|^N \dt \right\}
\end{align*}
and hence for $\om \in S^{N-1}$, we have
\begin{align*}
|\psi(r \om)|^N \leq 2^{N-1} \left\{ |\psi(\om)|^N + (\log r)^{N-1} \int_1^r t^{N-1} |\Gr \psi(t\om)|^N \dt \right\}.
\end{align*}
Now multiply both sides by $r^{N-1} \g(r)$ and integrate over $S^{N-1} \times (1, \infty)$ to get
\begin{align}\label{ce 1}
\no  \frac{1}{2^{N-1}} \displaystyle \int_{1}^{\infty} & \int_{S^{N-1}} |\psi(r \om)|^N r^{N-1}  \g(r) \dS \; \dr \leq  \int_1^{\infty} \int_{S^{N-1}}  | \psi(\om)|^N r^{N-1} \g(r) \dS \; \dr \\
& + \left( \int_{1}^{\infty} (r \log r)^{N-1} \g(r) \dr \right) \left( \int_1^{\infty} \int_{S^{N-1}}  r^{N-1} |\Gr \psi(r\om)|^N   \dS \; \dr \right).
\end{align}
Using  the trace embedding of $\wNb$ into $L^N(\pa B^c_1)$ (Theorem 2.86 of \cite{Demengel}, page  100) we estimate the following integral:
\begin{align}\label{critical 2}
\no \int_1^{\infty} \int_{S^{N-1}} | \psi(\om)|^N r^{N-1} \g(r) & \dS  \; \dr  = \left( \int_1^{\infty} r^{N-1} \g(r) \dr \right) \left( \int_{S^{N-1}} |\psi(\om)|^N \dS\right) \\
 & \leq C_1 \left( \int_1^{\infty} r^{N-1} \g(r) \dr \right)\disb \left(|\psi(x)|^N + | \Gr \psi(x)|^N \right) \dx,
\end{align}
where $C_1=C_1(N)>0$ is the embedding constant. 
By combining the above inequalities and using the density argument we obtain
\begin{align*}
\disb \g(|x|) |\psi(x)|^N  \dx  \leq C  \left( \int_1^{\infty} (r + r \log r)^{N-1} \g(r) \dr \right) \| \psi \|^N_{\wNb}, \quad \forall \psi \in \wNb.
\end{align*}
\end{proof}
\begin{theorem}\label{weighted Hardy supercritical}
Let $N < p$. If $\g \in L^1((1, \infty);r^{p-1}),$ then there exits $C = C(N,p)>0$ such that  
\begin{align}\label{Hardy supercritical}
\disb \g(|x|)|\phi(x)|^p  \dx \leq C \Vert \g \Vert_{L^1((1, \infty) ; r^{p-1})} \|\phi \|^p_{\wpb}, \quad \forall \phi \in \wpb. 
\end{align}
\end{theorem}
\begin{proof}
For $\vph$ as in the above proof we have
\begin{align*}
\vph(r) - \vph(1)  = \int^r_1 \vph'(t) \dt = \int^r_1 t^{\frac{1-N}{p}} t ^{\frac{N-1}{p}} \vph'(t) \dt.
\end{align*}
By H\"{o}lder inequality we get
\begin{align*}
\abs{\vph(r) - \vph(1)} \leq  \left( \int_1^r t^{\frac{1-N}{p-1}}  \dt \right)^{\frac{1}{\p}} \left( \int_1^r t^{N-1} | \vph'(t)|^p \dt \right)^{\frac{1}{p}}, 
\end{align*}
where $\p$ is the H\"older conjugate of $p.$ Thus
\begin{align}\label{super 1}
|\vph(r)|^p \leq 2^{p-1} \left\{ |\vph(1)|^p + \left( \frac{p-1}{p-N} \right)^{p-1} r^{p-N}  \left( \int_1^r t^{N-1} | \vph'(t)|^p \dt \right) \right\}. 
\end{align}
As before, we multiply both sides of \eqref{super 1} by $r^{N-1} \g(r)$ and integrate over $S^{N-1} \times (1, \infty)$ to get
\begin{align*}
\no  \frac{1}{2^{p-1}} \displaystyle \int_{1}^{\infty} \int_{S^{N-1}} & |\psi(r\om)|^p  r^{N-1} \g(r) \dS \; \dr \leq \int_1^{\infty} \int_{S^{N-1}} | \psi(\om)|^p r^{N-1} \g(r) \dS \; \dr \\
& + \left( \frac{p-1}{p-N} \right)^{p-1} \left( \int_1^{\infty} r^{p-1} \g(r) \dr \right)  \left( \int_1^{\infty} \int_{S^{N-1}} r^{N-1} | \Gr \psi(r\om) |^p   \dS \; \dr \right).
\end{align*}
The rest of the proof follows as in the proof of Theorem \ref{weighted Hardy critical}.
\end{proof}
Next, we prove the embeddings given above  are indeed compact. 

\begin{theorem}\label{compact embedding sub sup}
Let $\g\in X$. Then $\wpb$ embedded compactly into $L^p(B_1^c,|g|).$
\end{theorem}

\begin{proof}
Let $\phi_n \rightharpoonup \phi$ in $\wpb.$ 
Set $ M = \sup \{ \| \phi_n - \phi \|_{\wpb} \}.$ Let $\ep > 0$ be arbitrary. By density of $C_c^{\infty}((1, \infty))$ in $X$,  there exits $\g_{\ep} \in C_c^{\infty}((1, \infty))$ such that 
$
\| \g - \g_{\ep} \|_{X} < \frac{\ep}{M^p}.
$
Now,
\begin{align}\label{compact4}
\disb |g| |\phi_n - \phi|^p \leq \disb \g |\phi_n - \phi|^p  \le   \disb  |\g - \g_{\ep}| |\phi_n - \phi|^p+\disb |\g_{\ep}| |\phi_n - \phi|^p .
\end{align}
 From Theorem \ref{weighted Hardy subcritical}-\ref{weighted Hardy supercritical},  we have  
\begin{align}\label{compact5}
\disb |\g - \g_{\ep}| |\phi_n - \phi|^p  & \leq  C \| \g - \g_{\ep} \|_{X} \| \phi_n - \phi \|_{\wpb}^p,
\end{align} 
where $C > 0$ is the embedding constant. By the compactness of the embedding of $\wpb$ into $L^p_{loc}(B^c_1)$, there exits $n_0 \in \N$ such that
$\int_{B^c_1} |\g_{\ep}| |\phi_n - \phi|^p  < \ep, \forall\, n > n_0.$
Now by the above inequalities we obtain $$\disb \g |\phi_n - \phi|^p < C\ep, \quad \forall\, n\ge n_0.$$ Thus $\phi_n$ converges strongly in $L^p(B_1^c,|g|)$ as required.
\end{proof}

\section{The variational settings}
Now we develop the functional settings for proving our main theorems. For $g$ as in Theorem \ref{Lorentz space} or Theorem \ref{Exterior Ball}, we consider the following functionals on $\wpb$:
$$J(\phi)  = \disb |\Gr \phi|^p;\quad G(\phi)=\disb g |\phi|^p.$$
One can easily verify that $J, G\in C^1(\wpb;\R)$ and for $ \phi,u \in \wpb$,
$$
\left< J'(\phi), u \right> = p\disb |\Gr \phi|^{p-2}\Gr \phi\cdot \Gr u;  \quad\quad  \left< G'(\phi), u \right> = p\disb g| \phi|^{p-2} \phi u,
$$
where $\left<\, ,\, \right>$ denotes the duality action.
\begin{definition}
 We say a function $g$ belongs to the class $\A$, if 
  $g\in L^1_{loc}(B_1^c)$, supp$(g^+)$ has a positive measure and  \begin{enumerate}[(i)]
                                \item $g\in \F_{\frac{N}{p}}$ with $N>p,$ or 
                                \item $\g \in \left\{\begin{array}{ll}
                    L^1((1, \infty); r^{p-1}), \, & N \neq p,\\
                     L^1((1, \infty); {(r(1 + \log r))}^{N-1}), \, &N=p.
                   \end{array}\right.$
                               \end{enumerate}
\end{definition}
\begin{proposition}\label{Compact}If $g\in \A$, then $G$ and $G'$ are compact on $\wpb$.  
\end{proposition}
\begin{proof}
\noi{\it Compactness of $G$:} If $g \in \F_{\frac{N}{p}}$, then $G$ is compact using the density of $C_c^{\infty}(B^c_1)$ in $\F_{\frac{N}{p}}$ and the arguments as in the proof of Theorem \ref{compact embedding sub sup}. If $\g \in X,$ then the compactness of $G$ follows from Theorem \ref{compact embedding sub sup}. \\
\noi{\it Compactness of $G'$:} Let $g\in \F_{\frac{N}{p}}$ and let $\phi_{n} \rightharpoonup \phi$ in $\wpb$.  For $v \in \wpb,$
\begin{align*}
\no  | \left< {G}^{\prime}(\phi_{n}) - {G}^{\prime}(\phi)), v \right> | & \leq \disb |g| |( | \phi_n|^{p-2} \phi_{n} - |\phi|^{p-2} \phi ) | |v| \\
& \leq   \left( \disb |g| | ( | \phi_n|^{p-2} \phi_{n} - |\phi|^{p-2} \phi )|^{\p} \right)^{ \frac{1}{\p}} \left( \disb |g| |v|^p \right)^{\frac{1}{p}} \\
& \leq C \left( \disb |g| | ( | \phi_n|^{p-2} \phi_{n} - |\phi|^{p-2} \phi )|^{\p} \right)^{ \frac{1}{\p}} \|g \|^{\frac{1}{p}}_{(\frac{N}{p}, \infty)} \| v \|_{\wpb},
\end{align*}
where $C > 0$ is the embedding constant. Therefore, $$ \| {G}^{\prime}(\phi_{n}) - {G}^{\prime}(\phi) \| \leq C \left( \disb |g| | ( | \phi_n|^{p-2} \phi_{n} - |\phi|^{p-2} \phi )|^{\p} \right)^{ \frac{1}{\p}} \|g \|^{\frac{1}{p}}_{(\frac{N}{p}, \infty)}. $$ Now consider the map  $K$ defined on $\wpb$ as $K(\phi)=|g|^{\frac{1}{\p}}|\phi|^{p-2}\phi.$ Clearly $K$ maps $\wpb$ into $L^{\p}(B^c_1)$ and  using a similar set of arguments as in the proof of Theorem \ref{compact embedding sub sup}, one can prove that  $K$ is compact. Hence we conclude $\| {G}^{\prime}(\phi_{n}) - {G}^{\prime}(\phi) \| \rightarrow 0,$ as $n\ra \infty.$ For $\g \in X$, the proof is similar.
 \end{proof}
For $g$ as before, consider the set  $$N_g:=  \{ \phi \in \wpb : \int_{B^c_1} g |\phi|^p = 1 \}= G^{-1}(1).$$ 
Since $g^+ \not \equiv 0,$ one can show that the set $N_g$ is nonempty (Proposition 4.2 of \cite{Kawohl-Lucia-Prashanth}). The functional $J$ is not coercive on $\wpb.$ We prove a Poincaré type inequality for functions in $N_g$ that will ensure $\int_{B_1^c} |\Gr \phi|^p$ is coercive on $N_g$. 
\begin{Lemma}\label{Poincare}
Let $g\in \A.$ Then there exits $m > 0$ such that
\begin{align}\label{Poincare inequality}
\disb |\Gr \phi|^p \geq m \disb |\phi|^p, \quad \forall \phi \in N_g.
\end{align}
\end{Lemma}
\begin{proof}
The proof is by contradiction. If \eqref{Poincare inequality} is not true, then there exits a sequence $( \phi_{n} )$ in $N_g $ such that
\begin{align}\label{Poincare 1}
\disb | \phi_n|^{p} = 1, \disb |\Gr \phi_n|^p \leq \frac{1}{n}.
\end{align}
Thus $( \phi_{n} ) $ is bounded in $\wpb $ and hence by the reflexivity we get a subsequence $(\phi_{n_k})$ of $(\phi_n)$ such that as $k \rightarrow \infty$, $\phi_{n_k} \rightharpoonup \phi $ in $\wpb$. Thus $\phi\in N_g,$ as $G$ is compact. Further, from \eqref{Poincare 1}, by weak lowersemicontinuity of $\|.\|_{p}$ and $J$,  we have: $$ \disb | \phi|^{p} \le  1 \text{ and } \int_{B^c_1} | \Gr \phi |^p = 0. $$ Now the connectedness yields $\phi = 0$, a contradiction as $\phi \in N_g$.
\end{proof}

\begin{remark}\label{manifold}
 For $\phi\in N_g,$ $\left< G'(\phi), \phi\right>=p.$ Thus 1  is a regular point of $G$ and hence $N_g$ gets a $C^1$ manifold structure.
For $\phi \in N_g,$ the tangent space at $\phi$ is given by (Proposition 4.3.33 and Remark 4.3.40 of \cite{Drabek-Milota})  $$ T_{\phi} N_g = \mbox{Ker}( G^{\prime}(\phi)).$$ 
Further, 
\begin{equation}\label{derivative}
  \Vert \dJ(\phi) \Vert = \sup_{\substack{v \in \text{Ker}( G^{\prime}(\phi)) \\ \| v \| = 1}} \left< J^{\prime}(\phi),v \right> =  \min_{\la \in \R} \Vert J^{\prime}(\phi) - \la {G}^{\prime}(\phi) \Vert
\end{equation}
(Proposition 6.4.35 of \cite{Drabek-Milota}). In particular, if $\phi$ is a critical point of $J$ on $N_g$, then $\phi$ is an eigenfunction of \eqref{Eqn:Neumann} corresponding to the eigenvalue $J(\phi).$
\end{remark}

\begin{definition}
 We say a map $J \in C^1(Y, \R)$ satisfies Palais-Smale (P. S.) condition on a $C^1$ manifold $ M  \subset Y$, if $(\phi_n)$ in $M$ such that $J(\phi_n) \rightarrow c \in \R$ and $ \Vert \dJ(\phi_n) \Vert \rightarrow 0,$ then $(\phi_n)$ has a subsequence that converges in $M$.
\end{definition}
\begin{Lemma}\label{P.S.} Let $g\in \A.$ Then $J$ satisfies the P. S. condition on $N_g$.
\end{Lemma}
\begin{proof}
Let $(\phi_n )$ be a sequence in $N_g$ such that $J(\phi_n) \rightarrow \la $ and $ \| \dJ(\phi_n) \| \rightarrow 0.$ For $\la \in \R,$ set $A_{\la} = J^{\prime} - \la G'.$ Then by \eqref{derivative}, there exits a sequence $( \la_n )$ such that 
\begin{align}\label{P.S. 1}
A_{\la_n}(\phi_n)=J^{\prime}(\phi_n) - \la_n {G}^{\prime}(\phi_n) \rightarrow 0, \quad \mbox{as} \; n\rightarrow \infty.
\end{align}
Using Lemma \ref{Poincare} and by the reflexivity of $\wpb$, up to a subsequence, $\phi_n \rightharpoonup \phi $ in $\wpb$. Further,
$
\inpr{J^{\prime}(\phi_n) - \la_n G^{\prime}(\phi_n),\phi_n} = p \left( J(\phi_n) - \la_n \right).
$
Thus $ \la_n  \rightarrow \la.$ Observe that  
$$ \inpr{ J^{\prime}(\phi_n) -  J^{\prime}(\phi), \phi_n-\phi}=\inpr{A_{\la_n}(\phi_n)-A_{\la}(\phi), \phi_n-\phi}+ \inpr{\la_n G^{\prime}(\phi_n)- \la G^{\prime}(\phi), \phi_n-\phi}.$$ 
Form the weak convergence of $(\phi_n)$ and the compactness of $G'$, we get 
$$ \inpr{ J^{\prime}(\phi_n) -  J^{\prime}(\phi), \phi_n-\phi}\ra 0, \quad \text{ as } n\ra \infty.$$
Now 
\begin{align*}
\no \frac{1}{p}\left< J^{\prime}(\phi_n) -  J^{\prime}(\phi), \phi_n - \phi \right> & = \Vert \Gr \phi_n \Vert_p^p + \Vert \Gr \phi \Vert_p^p - \disb |\Gr \phi_n|^{p-2}\Gr \phi_n \cdot \Gr \phi - \disb |\Gr \phi|^{p-2}\Gr \phi \cdot \Gr \phi_n \\
\no & \geq  \Vert \Gr \phi_n \Vert_p^p + \Vert \Gr \phi \Vert_p^p - \Vert \Gr \phi_n \Vert_p^{p-1} \Vert \Gr \phi \Vert_p - \Vert \Gr \phi \Vert_p^{p-1} \Vert \Gr \phi_n \Vert_p  \\
& = \left( \Vert \Gr \phi_n \Vert_p^{p-1} - \Vert \Gr \phi \Vert_p^{p-1} \right) \left( \Vert \Gr \phi_n \Vert_p - \Vert \Gr \phi \Vert_p \right).
\end{align*}
Hence $\Vert \Gr \phi_n \Vert_p \rightarrow \Vert \Gr \phi \Vert_p.$ Thus the weak convergence of $(\phi_n)$ in $\wpb$ and the uniform convexity of $(L^p(B_1^c))^N$ gives $\Gr \phi_n \rightarrow \Gr \phi$ in $(L^p(B_1^c))^N$. Now using Lemma \ref{Poincare}, we conclude that $\phi_n \rightarrow \phi $ in $\wpb.$ This completes the proof.
\end{proof}
\section{Proof of the main theorems} \label{existence and properties of eigenvalue}
\noi{\bf Proof of Theorem \ref{Lorentz space} and \ref{Exterior Ball}}:

\noi{\it The existence:} 
 Recall that  $$\la_1 = \inf_{\phi \in N_g} \disb |\Gr \phi|^p.$$ Let $ ( \phi_{n} )$ be a minimizing sequence for $J$ on $N_g$.   As before, using Lemma \ref{Poincare} we get sequence $( \phi_{n} )$ is bounded in $ \wpb.$ Thus by the reflexivity, $(\phi_{n})$ has a subsequence $(\phi_{{n}_k})$ that  converges weakly to some $\Phi \in \wpb$. Since the set $N_g$ is weakly closed by the compactness of $G$ proved before, $\Phi \in N_g.$ Further, by weak lowersemicontinuity of $J$, $$ \la_1 \leq J(\Phi) \leq \liminf J(\phi_{n_k}) = \la_1. $$ Thus $\la_1$ is attained and hence $\Phi$ is a critical point of $J$ on $N_g.$ Therefore, from Remark \ref{manifold} we see that $\la_1$ is an eigenvalue of \eqref{Eqn:Neumann} and $\Phi$ is an eigenfunction corresponding to $\la_1.$ 

\noi{\it The principality:} Clearly $|\P|$ is also an eigenfunction corresponding to $\la_1$.  Thus for $ v \in C_c^{\infty}(B^c_1)$ with $v\ge 0,$ 
\begin{align*}
\disb |\Gr (|\P|)|^{p-2} \Gr(|\P|) \cdot \Gr \upsilon + \la_1 \disb g^{-} |\P|^{p-1}v = \la_1 \disb g^{+} |\P|^{p-1}\upsilon \geq 0.
\end{align*}
Using H\"{o}lder inequality, one can verify that $g^{-}|\P|^{p-1} \in L_{loc}^1(B^c_1).$ Thus $|\P|$ satisfies all the conditions of Proposition \ref{STM}. Hence $|\P| > 0$ a.e. in $B^c_1.$ 

\noi {\it The uniqueness and the simplicity:} The uniqueness of the principal eigenvalue  can be obtained using the Picone's identity (Theorem 1.1 of \cite{Picone}). The simplicity follows using  the same arguments as in Theorem 1.3 of \cite{Kawohl-Lucia-Prashanth}.

\noi{\it Isolatedness:} Suppose  $(\la_n)$ is a sequence of eigenvalues of \eqref{Eqn:Neumann} converging to $\la_1$. For each $n$, let $ \phi_n\in N_g$ be an eigenfunction corresponding to  $ \la_{n}$. Then $J(\phi_n) = \la_n$ and for each $n \in \N$, we have
\begin{align*}
\left< J^{\prime}(\phi_n) - \la_n G^{\prime}(\phi_n), \upsilon \right> & =  \disb |\Gr \phi|^{p-2} \Gr \phi_n \cdot \Gr \upsilon - \la_n \disb g|\phi_n|^{p-2} \phi_n \upsilon = 0,
\end{align*}
i.e., $ \| \dJ(\phi_n)\| = 0.$ Hence using Lemma \ref{P.S.} we conclude that $\phi_n \rightarrow \pm |\Phi|$ in $\wpb.$ Assume that $\phi_n \rightarrow |\Phi|.$ Thus by Egorov's theorem there exits $A \subset B^c_1$ with $ | A | > 0$ and $\phi_n $ converges to $|\Phi|$ uniformly on $A.$ Thus there exits $n_0 \in \N$ such that for all $n \geq n_0,\; \phi_n^- = 0$ a.e. in $A.$ Further, from \eqref{weak}, $$\disb |\Gr \phi_n^-|^p = \la_n \disb g |\phi_n^-|^p.  $$ For $v_n = \left( \int_{B^c_1} g |\phi_n^-|^p \right)^{-\frac{1}{p}}{\phi_n^-},$ observe that $ \int_{B^c_1} g |v_n|^p = 1 \; \mbox{and} \; \int_{B^c_1} |\Gr v_n|^p=\la_n \rightarrow \la_1. $ Therefore, $v_n \rightarrow |\Phi|$. A contradiction, as $v_n = 0$ a.e. in $A$ for $n\ge n_0.$ Thus such a sequence $(\la_n)$ does not exists.
\qed
\begin{remark} 
\begin{enumerate}[(a)]
 \item For $g \in \A$, we have $\la_1 > 0$. Hence $\frac{1}{\la_1}$ is the best constant in the following Hardy-Sobolev inequality $$ \disb g(x) |\phi(x)|^p \dx \leq C \disb |\Gr \phi(x)|^p \dx, \quad \forall \phi \in \wpb$$  and it is attained. 
 \item  For $N > p,$ Theorem \ref{Lorentz space} holds for any unbounded domain in $\R^N$ and holds for any bounded domain with the additional assumption $\int g \dx < 0$.  Since  $L^{\frac{N}{p}}(B^c_1)$ is strictly contained in $\F_{\frac{N}{p}}$ (Proposition 3.1 of \cite{AMM}), Theorem \ref{Lorentz space}  with the additional assumption $\int g \dx < 0,$ extends  the results of \cite{BrownLin,  HabibTsouli, Huang}. 
\item 
 The spaces $\F_{\frac{N}{p}}$ and $L^1((1, \infty);r^{p-1})$ are not comparable. For $N > p,$ we consider the following two functions: 
\begin{align*}
g_1(x) & = \displaystyle \frac{1}{|x|^q}, \quad \text{for} \ p < q < \infty;  \quad g_2(x) = \left\{ \begin{array}{lr}
			 0, \;  & \ 1 \leq |x| \leq 2 ; \\
			 \displaystyle \frac{1}{|x|^p \log |x|}, & \; |x| > 2. \\
			 \end{array}
		\right. 
\end{align*} 
The function $g_1 \in L^1((1, \infty);r^{p-1})$ but does not belong to $L^{\frac{N}{p}, \infty}(B^c_1),$ whereas $g_2 \in L^{\frac{N}{p}}(B^c_1)$ but does not belong to $L^1((1, \infty);r^{p-1}).$
\end{enumerate}
\end{remark}

\subsection{The existence of an infinite set of eigenvalues}\label{set of infinite eigenvalues}
For the existence of a sequence of eigenvalues of \eqref{Eqn:Neumann}, we use the Ljusternik-Schnirelmann theory on $C^1$ manifold due to Szulkin \cite{Szulkin}. Let $A$ be a closed symmetric (i.e. $-A = A$) subset of a $C^1$ manifold $M$. The {\it krasnoselski} genus $\ga(A)$ is defined to be the smallest integer $k$ for which there exists a non-vanishing odd continuous mapping from $A$ to $\R^k$. For more details of genus we refer to \cite{Rabinowitz}. The next theorem follows from Corollary 4.1 of \cite{Szulkin}. 
\begin{theorem}[Szulkin's Theorem]\label{Szulkin} 
Let $M$ be a closed symmetric $C^1$ submanifold of a real Banach space $X$ and $0 \not \in M$. Let $f \in C^1(M, \R)$ be even and bounded below. Let $$ \la_j = \inf_{A \in \Ga_j} \sup_{\phi \in A} f(\phi),  $$
and $\Ga_j = \{ A \subset M : A \; \text{compact, symmetric and } \; \ga(A) \geq j \}. $ If $\Ga_n \neq \emptyset$ for some $n \geq 1$ and $f$ satisfies the $\text{(P. S.)}_{c}$ condition for all $c = \la_j,$ where $j = 1,2, ..., n$,  then $\la_j$ are the critical values of $f$.
\end{theorem}
\noindent \textbf{Proof of Theorem \ref{infinite eigenvalues}}.\\
\noi{\it (i)}  The set $N_g$ and the functional $J$ satisfy all the properties of Szulkin's theorem. Using the arguments as in the proof of Lemma 5.9 of \cite{Anoop}, one can show that,  for each $n \in \N,$ the set $\Ga_n = \{ A \subset N_g : A \; \text{compact, symmetric and } \; \ga(A) \geq n \}$ is nonempty.  Hence by Theorem \ref{Szulkin}, there exits $\phi_n \in N_g$ such that $\| \dJ(\phi_n) \|= 0$ and $J(\phi_n) = \la_n$. Therefore, $\la_n$ is an eigenvalue of \eqref{Eqn:Neumann} and $\phi_n$ is an eigenfunction corresponding to $\la_n$. Further, ($\la_n) $ is unbounded by the same arguments as in the proof of  Theorem 2 of \cite{Huang1995}.

\noi{\it (ii)} Let $(\la_{n})$ be a sequence of eigenvalues of \eqref{Eqn:Neumann} such that $\la_{n} \rightarrow \la$. Let $ \phi_n $ be an eigenfunction corresponding to $ \la_{n} $ satisfying $\int_{B^c_1} g |\phi_n|^p = 1$. Thus $J(\phi_n)=\la_n$ and $\| \dJ(\phi_n) \|= 0$. Hence  by Lemma \ref{P.S.}, there exists  a subsequence of  $(\phi_n)$  that converges to $\phi$ in $\wpb.$ Now the continuity of $J'$ and $G'$ ensures that $\la$ is an eigenvalue of \eqref{Eqn:Neumann}.
\qed

\bibliographystyle{plain}
\bibliography{Ref}
\noi {\bf  T. V.  Anoop } \\  Department of Mathematics,\\   Indian Institute of Technology Madras, \\ Chennai, 600036, India. \\ 
{\it Email}:{ anoop@iitm.ac.in} \\
		\noi {\bf  Nirjan Biswas } \\  Department of Mathematics,\\   Indian Institute of Technology Madras, \\ Chennai, 600036, India. \\ 
{\it Email}:{ nirjaniitm@gmail.com}
\end{document}